\documentclass[12pt]{amsart}

\usepackage{latexsym}
\usepackage{amsfonts}
\usepackage{amsmath}
\usepackage{amsthm}
\usepackage{amssymb}
\usepackage{amscd}
\usepackage{hyperref}
\usepackage{enumerate}
\usepackage[english]{babel}
\usepackage{cleveref}
\usepackage[usenames]{color}
\newcommand{\Supp}{\operatorname{Supp}}

\newcommand{\NN}{\mathcal{N}}

\newcommand{\gotG}{\mathfrak{g}}
\newcommand{\gotU}{\mathfrak{u}}

\newcommand{\proofend}{\hfill$\Box$\smallskip}
\providecommand{\cS}{\mathcal{S}}
\newcommand{\Sc}{\cS}
\newtheorem{thm}{Theorem}[section]
\newtheorem{thm*}{Theorem}
\newtheorem{lem}[thm]{Lemma}
\newtheorem{lem*}[thm*]{Lemma}
\newtheorem{prop}[thm]{Proposition}
\newtheorem{cor}[thm]{Corollary}
\newtheorem{prop*}[thm*]{Proposition}
\newtheorem{cor*}[thm*]{Corollary}

\newtheorem{defn}[thm]{Definition}
\newtheorem{defn*}[thm*]{Definition}
\newtheorem{thm-defn}[thm]{Theorem-Definition}
\newtheorem{notn}[thm]{Notation}

\providecommand{\cO}{\mathcal{O}}

\providecommand{\g}{\mathfrak{g}}
\providecommand{\al}{\alpha}

\providecommand{\fh}{\mathfrak{h}}
\providecommand{\fg}{\mathfrak{g}}

\providecommand{\ft}{\mathfrak{t}}

\providecommand{\fv}{\mathfrak{v}}

\providecommand{\cF}{\mathcal{F}}
\providecommand{\cG}{\mathcal{G}}
\providecommand{\cN}{\mathcal{N}}
\providecommand{\cS}{\mathcal{S}}
 \newtheorem{introtheorem}{Theorem}
 
 \crefname{introtheorem}{theorem}{theorems}
 \Crefname{introtheorem}{Theorem}{Theorems}
  \newtheorem{introthm}[introtheorem]{Theorem}
   \crefname{introthm}{theorem}{theorems}
 \Crefname{introthm}{Theorem}{Theorems}

  \crefname{introcorollary}{corollary}{corollaries}
 \Crefname{introcorollary}{Corollary}{Corollaries}

 \newtheorem{introcor}[introtheorem]{Corollary}
   \crefname{introcor}{corollary}{corollaries}
 \Crefname{introcor}{Corollary}{Corollaries}
 
   \crefname{introconjecture}{conjectures}{conjectures}
 \Crefname{introconjecture}{Conjecture}{Conjectures}
 
    \crefname{introconj}{conjectures}{conjectures}
 \Crefname{introconj}{Conjecture}{Conjectures}

     \crefname{introlem}{lemma}{lemmas}
 \Crefname{introlem}{Lemma}{Lemmas}
 
 \crefname{introremark}{remark}{remarks}
 \Crefname{introremark}{Remark}{Remarks}
 
  \crefname{introrem}{remark}{remarks}
 \Crefname{introrem}{Remark}{Remarks}

   \crefname{introprop}{Proposition}{Propositions}
 \Crefname{introprop}{Proposition}{Propositions}

   \crefname{introdefn}{definition}{definitions}
 \Crefname{introdefn}{Definition}{Definitions}
 
   \crefname{intronotn}{notation}{notations}
 \Crefname{intronotn}{Notation}{Notations}
 
   \crefname{introtask}{task}{tasks}
 \Crefname{introtask}{Task}{Tasks}
 
  \crefname{introprob}{problem}{problems}
 \Crefname{introprob}{Problem}{Problems}
 
   \crefname{introquestion}{question}{questions}
 \Crefname{introquestion}{Question}{Questions}

 \crefname{theorem}{theorem}{theorems}
 \Crefname{theorem}{Theorem}{Theorems}
  \crefname{thm}{theorem}{theorems}
 \Crefname{thm}{Theorem}{Theorems}

  \crefname{corollary}{Corollary}{Corollaries}
 \Crefname{corollary}{Corollary}{Corollaries}

   \crefname{cor}{Corollary}{Corollaries}
 \Crefname{cor}{Corollary}{Corollaries}

   \crefname{conjecture}{conjectures}{conjectures}
 \Crefname{conjecture}{Conjecture}{Conjectures}

    \crefname{conj}{conjectures}{conjectures}
 \Crefname{conj}{Conjecture}{Conjectures}

     \crefname{lem}{lemma}{lemmas}
 \Crefname{lem}{Lemma}{Lemmas}

      \crefname{lemma}{Lemma}{Lemmas}
 \Crefname{lemma}{Lemma}{Lemmas}

 \crefname{remark}{remark}{remarks}
 \Crefname{remark}{Remark}{Remarks}

  \crefname{rem}{remark}{remarks}
 \Crefname{rem}{Remark}{Remarks}

   \crefname{rem}{remark}{remarks}
 \Crefname{rem}{Remark}{Remarks}

   \crefname{proposition}{Proposition}{Proposition}
 \Crefname{proposition}{Proposition}{Proposition}

    \crefname{prop}{Proposition}{Propositions}
 \Crefname{prop}{Proposition}{Propositions}

   \crefname{defn}{definition}{definitions}
 \Crefname{defn}{Definition}{Definitions}

   \crefname{notn}{notation}{notations}
 \Crefname{notn}{Notation}{Notations}

   \crefname{task}{task}{tasks}
 \Crefname{task}{Task}{Tasks}

  \crefname{prob}{problem}{problems}
 \Crefname{prob}{Problem}{Problems}

   \crefname{question}{question}{questions}
 \Crefname{question}{Question}{Questions}

\theoremstyle{remark}
\newtheorem{rem}[thm]{Remark}

\newcommand{\Dima}[1]{{{#1}}}
\newcommand{\Rami}[1]{{{#1}}}
\newcommand{\DimaA}[1]{{{#1}}}
\newcommand{\DimaB}[1]{{{#1}}}
\newcommand{\DimaC}[1]{{{#1}}}

\newcommand{\RamiA}[1]{{{#1}}}

\begin{document}

\author{Avraham Aizenbud}\thanks{A.A. partially supported by NSF grant DMS-1100943 and ISF grant 687/13}
\address{Avraham Aizenbud,
Faculty of Mathematics and Computer Science, Weizmann
Institute of Science, POB 26, Rehovot 76100, Israel }
\email{aizenr@gmail.com}
\urladdr{http://math.mit.edu/~aizenr/}
\author{Dmitry Gourevitch}\thanks{D.G. partially supported by ERC grant 291612, and a Minerva foundation grant.}
\address{Dmitry Gourevitch, Faculty of Mathematics and Computer Science, Weizmann
Institute of Science, POB 26, Rehovot 76100, Israel }
\email{dimagur@weizmann.ac.il}
\urladdr{http://www.wisdom.weizmann.ac.il/~dimagur}
\author{Alexander Kemarsky}\thanks{A.K. partially supported by ISF grant No. 1394/12}
\address{Alexander Kemarsky, Faculty of Mathematics , Technion - Israel
Institute of Technology, Haifa 32000, Israel }
\email{alexkem@tx.technion.ac.il}
\keywords{Wave-front set, spherical character} 
\subjclass[2010]{20G05, 20G25, 22E35, 46F99}
\date{\today}
\title[Vanishing of certain equivariant distributions]{Vanishing of certain equivariant distributions on $p$-adic spherical spaces, and non-vanishing of spherical Bessel functions}
\begin{abstract}
We prove vanishing of distribution on p-adic spherical spaces that are equivariant with respect to a generic character of the nilradical of a Borel subgroup and satisfy a certain condition on the wave-front set. We deduce from this non-vanishing of spherical Bessel functions for Galois symmetric pairs.
\end{abstract}

\maketitle
\section{Introduction}
{Let $\mathbf{G}$ be a  reductive group, quasi-split over a non-Archimedean local field $F$ of characteristic zero. Let $\mathbf{B}$ be a Borel
subgroup of  $\mathbf{G}$, and let $\mathbf{U}$ be the unipotent radical
of $\mathbf{B}$.
Let $\mathbf{H}$ be a closed subgroup of $\mathbf{G}$.
Let $G,B,U,H$ denote the $F$-points of $\mathbf{G},\mathbf{B},\mathbf{U},\mathbf{H}$ respectively.
Suppose that $\mathbf{H}$ is an $F$-spherical subgroup of $\mathbf{G}$,
i.e. that there are finitely many $B\times H$-double cosets in $G$.
Let $\fg,\fh$ be the Lie algebras of $G,H$ respectively.
%
Let $\psi$ be a non-degenerate character of $U$ and let $\chi$
be a \RamiA{(locally constant)} character of $H$. For $x \in G$ denote $H^x:=xHx^{-1}$ and denote by $\chi^x$ the character of $H^x$ defined by conjugation of $\chi$.
For a $B \times H$-double coset $\cO \subset G$ define
$$\cO_c:=\left \{x \in \cO \, |\, \left. \psi \right|_{H^{x} \cap U} =  \left. \chi^{x} \right|_{H^{x} \cap U} \right \}. $$
Let $$Z := \bigcup_{\cO \text{ s.t. }\cO\neq\cO_c } \cO.$$
Identify $T^*G$ with $G \times \fg^*$ and let $\cN_{\g^*}$ be the set of nilpotent elements in $\g^*$.

\RamiA{Consider the action of $U \times H$ on $G$ given by $(u,h)x=uxh^{-1}$. This gives rise to an action of $U \times H$ on the space $\cS(G)$ of Schwartz (i.e. locally constant compactly supported) functions on $G$  and the dual action on the space of  distributions $ \cS^*(G)$.}

In this paper we prove the following theorem.
\begin{introthm}[see \Cref{sec:PfMain}]\label{thm: main theorem}
Let $\xi \in \cS^*(G)^{(U\times H, \psi \times \chi)} $ be a\RamiA{n equivariant distribution}
on $G$\RamiA{, i.e. $(u,h)\xi=\psi(u)\chi(h)\xi$}. Suppose that the wave-front set (see \cref{subsec:WF})  $WF(\xi)$ lies in $G\times \NN_{\fg^*}$
and $\Supp(\xi) \subset Z$.
Then $\xi = 0$.
\end{introthm}

In the case when $\mathbf{H}$ is a subgroup of Galois type we can prove a stronger statement. By a subgroup of Galois type we mean a  subgroup $\mathbf{H}\subset \mathbf{G}$
such that $$(\mathbf{G}\times_{\mathrm{Spec} F}\mathrm{Spec} E,\mathbf{H}\times_{\mathrm{Spec} F}\mathrm{Spec} E)  \simeq (\mathbf{H}\times_{\mathrm{Spec} F} \mathbf{H}\times_{\mathrm{Spec} F}\mathrm{Spec} E,\Delta \mathbf{H}\times_{\mathrm{Spec} F}\mathrm{Spec} E) $$ for some field extension $E$ of $F$, where $\Delta \mathbf{H}$ is the diagonal copy of $\mathbf{H}$ in $\mathbf{H} \times_{\mathrm{Spec} F} \mathbf{H}$.

\begin{introcor}[see Section \ref{section: proof of cors}]\label{cor: galois case}
Let $\mathbf{H}\subset \mathbf{G}$ be a subgroup of Galois type, and let $\chi$ be a character of $H$.
Let $S $ be the union of all non-open $B\times H$-double cosets in $G$.
Let $\xi \in \cS^*(G)^{(U\times H,\psi\times \chi)} $. Suppose that $WF(\xi) \subset G \times \NN_{\fg^*}$
and $\Supp(\xi) \subset S$. Then $\xi = 0$.
\end{introcor}

%

Note that if $\chi$ is trivial, we can consider the distribution $\xi$ as a distribution on $G/H$. Considering $ \mathbf{\tilde G}:=\mathbf{G}\times \mathbf{G}$ and taking $\mathbf{H}$ to be the diagonal copy of $\mathbf{G}$ we obtain the following corollary for the group case.

\begin{introcor}[see Section \ref{section: proof of cors}]\label{cor: group case}
Let $\psi_1$ and $\psi_2$ be non-degenerate characters of $U$. Let $B \times B$ act on
$G$ by $(b_1,b_2)g := b_1gb^{-1}_2$. Let $S$ be the complement to the open $B \times B$-orbit
in $G$. For any $x\in G$, identify $T_xG$ \Dima{with} $\fg $ and $T_x^*G$ with $ \fg^*$.
Let $$\xi \in \cS^*(G)^{U\times U,\psi_1 \times \psi_2}$$ and suppose that
$WF(\xi) \subset S \times\ \cN_{\g^*}$.
Then $\xi = 0$.
\end{introcor}

\subsection{Applications to non-vanishing of spherical Bessel functions}

Let $\pi$ be an admissible representation of $G$ (of finite length)\DimaB{, and $\tilde \pi$  be the smooth contragredient representation}.
Let $\mathbf{H}\subset \mathbf{G}$ be an algebraic spherical subgroup and let $\chi$ be a character of $H$.
Let $\phi\in (\pi^*)^{(U,\psi)}$ be a $(U,\psi)$-equivariant functional on $\pi$ and $v$ be an $(H,\chi)$-equivariant functional on $\tilde \pi$. \DimaB{For any function $f\in \Sc(G)$, we have $\pi^*(f) \phi \in \tilde \pi \subset \pi^*$.

This enables us to
define the \emph{spherical Bessel distribution} corresponding to $v$ and $\phi$ by
$$\xi_{v,\phi}(f):=\langle  v, \pi^*(f) \phi \rangle. $$}
%
%
By \cite[Theorem A]{AGS} we have  $WF(\xi_{v,\phi}) \subset G\times \cN_{\fg^*}$.

  The \emph{spherical Bessel function} is defined to be the restriction $j_{v,\phi}:=\xi_{v,\phi}|_{G - S}$, where $S $ is the union of all non-open $B\times H$-double cosets in $G$. 
One can easily deduce from \Rami{ \cite[Theorem A]{AGS} 
and \Cref{lem:WF}}
that
$j_{v,\phi}$ is a smooth function.
Theorem \ref{thm: main theorem} and Corollary \ref{cor: galois case} imply the following corollary.
\begin{introcor}\label{cor:Bessel}
Suppose that  $\pi$ is irreducible and
 $v,\phi$ \Dima{are non-zero}. Then
\begin{enumerate}[(i)]
\item \DimaB{For any open subset $U\subset G$ that includes $G \setminus Z$ we have $\xi_{v,\phi}|_{U} \neq 0$.}
\item \label{it:Gal} If $\mathbf{H}$ is a subgroup of Galois type then $j_{v,\phi} \neq 0$.
\end{enumerate}
\end{introcor}

For the group case  this corollary was proven in \cite[Appendix A]{LM}.
}


\subsection{Related results}

In \cite{AG}  a certain  Archimedean analog of Theorem \ref{thm: main theorem} is proven (see \cite[Theorem A]{AG}). This analog implies that the Archimedean analog  of Corollary \ref{cor:Bessel}\eqref{it:Gal} holds for any spherical pair $(G,H)$ (see \cite[Corollary B]{AG}).

Corollary \ref{cor: group case} together with  \cite[Theorem A]{AGS} can replace \cite[Theorem 3]{GK} in the proof of uniqueness of Whittaker models \cite[Theorem C]{GK}.

\Cref{thm: main theorem} can be used in order to study the dimensions of the spaces of $H$-invariant functionals on irreducible generic representations of $G$ (see \cite[\S 1.3]{AG} for more details). It can also be used in the study of analogs of Harish-Chandra's density theorem (see \cite[\S 1.7]{AGS} for more details).

\subsection{Acknowledgements}

We would like to thank \Rami{Moshe Baruch}, Shachar Carmeli, Guy Henniart, Friedrich Knop, Erez Lapid, Andrey Minchenko, \Rami{Eitan Sayag, Omer Offen,}  and Dmitry Timashev for fruitful discussions.

%
%

\section{Preliminaries}\label{sec:prel}

\subsection{Conventions}

\begin{itemize}
\item We fix $F,\mathbf{G},\mathbf{B},\mathbf{U},\mathbf{X}$ and $\psi$ as in the introduction.

\item All the algebraic groups and algebraic
varieties that we consider  are defined over $F$.
We will use capital bold letters, e.g. $\mathbf{G},\mathbf{X}$ to denote algebraic groups and varieties defined over $F$, and their non-bold versions to denote the $F$-points of these varieties, considered as $l$-spaces or $F$-analytic manifolds (in the sense of \cite{Ser}).

 \item When we use a capital Latin letter to denote an $F$-analytic group or an  algebraic group, we use the corresponding Gothic letter to denote its Lie algebra.


\item We denote by $G_x$ the stabilizer of $x$ and by $\fg_x$ its Lie algebra.
\DimaC{
\item For an $F$-analytic manifold $X$, a  submanifold $Y \subset X$ and a point $y\in Y$ we denote by $CN_Y^X\subset T^*X$ the conormal bundle to $Y$ in $X$, and by $CN_{Y,y}^X$ the conormal space at $y$ to $Y$ in $X$.
\item By a smooth measure on an $F$-analytic manifold we mean a measure which in a neighborhood of any point coincides (in some local coordinates centered at the origin) with some Haar measure on a closed ball centered at 0. A Schwartz measure is a compactly supported smooth measure.
\item The space of generalized functions $\cG(X)$ on an $F$-analytic manifold $X$ is defined to be the dual of the space of Schwartz measures. One can identify  $\cG(X)$ with $\Sc^*(X)$ by choosing a smooth measure with full support.
\item Let  $\phi:X\to Y$ be  a submersion of analytic manifolds. Note that the pushforward of a Schwartz measure with respect to $\phi$ is a Schwartz measure. By dualizing the pushforward map we define the pullback map $\phi^*:\cG(Y)\to \cG(X)$.
}
\end{itemize}

\subsection{Wave front set}\label{subsec:WF}

In this section we give an overview of the theory of the wave front set as developed  by D.~Heifetz \cite{Hef}, following L.~H\"ormander (see \cite[\S 8]{Hor}).
For simplicity we ignore here the difference between distributions and generalized functions.

\begin{defn}\label{def:wf}$ $
\begin{enumerate}
\item
Let $V$ be a finite-dimensional vector space over $F$.
Let $f \in C^{\infty}(V^*)$ and $w_0 \in V^*$. We  say that $f$ \emph{vanishes asymptotically in the direction of} $w_0$
if there exists $\rho \in \Sc(V^*)$ with $\rho(w_0) \neq 0$ such that the function $\phi \in C^\infty(V^* \times F)$ defined by $\phi(w,\lambda):=f(\lambda w) \cdot \rho(w)$ is compactly supported.

\item
Let $U \subset V$ be an open set and $\xi \in \Sc^*(U)$. Let $x_0 \in U$ and $w_0 \in V^*$.
We say that $\xi$ is \emph{smooth at} $(x_0,w_0)$ if there exists a compactly supported non-negative function $\rho \in \Sc(V)$ with $\rho(x_0)\neq 0$ such that the Fourier transform $\cF^*(\rho \cdot \xi)$ vanishes asymptotically in the direction of
$w_0$.
\item
The complement in $T^*U$
of the set of smooth pairs $(x_0, w_0)$ of $\xi$ is called the
\emph{wave front set of} $\xi$ and denoted by $WF(\xi)$.

\item For a point $x\in U$ we denote $WF_x(\xi):=WF(\xi)\cap T^*_xU$.

\end{enumerate}
\end{defn}

\begin{rem}$ $
\begin{enumerate}
\item Heifetz  defined  $WF_{\Lambda}(\xi )$ for any open subgroup $\Lambda$ of $F^{\times}$ of finite index. Our definition above differs  slightly  from the definition in \cite{Hef}. They relate by
\begin{equation*}
WF(\xi)-(U \times \{0\})= WF_{F^{\times}}(\xi).
\end{equation*}
\item \RamiA{Though  the notion of Fourier transform depends on a choice of a non-degenerate additive character of $F$, this dependence effects the Fourier transform only by dilation, and thus does not change our notion of wave front set.}
\end{enumerate}
\end{rem}



\Rami{
\begin{prop}
[{see  \cite[Theorem 8.2.4]{Hor}  and \cite[Theorem 2.8]{Hef}}]\label{prop:SubPull}
\label{submrtion}
Let $U \subset F^m$ and $V \subset F^n$ be open subsets, and
suppose that  $\phi: U \to V$ is an analytic submersion. Then for any
$\xi \in \Sc^*(V)$, we have $$WF(\phi^*(\xi)) \subset \phi^*(WF(\xi)):=\left\{ (x,v) \in T^*U \vert \,  \exists w\in WF_{\phi(x)}(\xi) \text{ s.t. } d_{\phi(x)}^*\phi(w)=v \right \} .$$
\end{prop}}

\begin{cor}\label{cor:SubPull}
Under the assumption of \Cref{prop:SubPull}
we have $$WF(\phi^*(\xi)) = \phi^*(WF(\xi)).$$
\end{cor}

\begin{proof}
The case when $\phi$ is an analytic diffeomorphism follows immediately from \Cref{prop:SubPull}. This implies the case of open embedding. It is left to prove the case of linear projection $\phi:F^{n+m}\to F^n$.  In this case the assertion follows from the fact that $\phi^*(\xi)=\xi\boxtimes 1_{F^m}$ where $ 1_{F^m}$ is the constant function 1 on $F^m$.
\end{proof}

\DimaB{This corollary enables to define the wave front set of any distribution on an $F$-analytic manifold, as a subset of the cotangent bundle. The precise definition follows.
\begin{defn}\label{def:bundle}
Let $X$ be an $F$-analytic manifold and $\xi \in \Sc^*(X)$. We define the wave front set $WF(\xi)$ as the set of all $(x,\lambda) \in T^*X$ which lie in the wave front set of $\xi$ in some local coordinates. In other words, $(x,\lambda) \in WF(\xi)$ if there exist open subsets $U\subset X$ and $V\subset F^n$, an analytic diffeomorphism $\phi:U \simeq V$ and $(y,\beta) \in T^*V$ such that $x\in U, \, \phi(x)=y, d_x\phi^*(\beta)=\lambda$, and $(y,\beta) \in WF((\phi^{-1})^*(\xi|_U))$.
\end{defn}}
%
\begin{thm}\label{thm: equivar_distr}(Corollary from \cite[Theorem 4.1.5]{A})
Let an $F$-analytic group $H$ act on an $F$-analytic manifold $Y$ \Dima{and}
let $\chi$ be a character of $H$.
Let $\xi  \in \Sc^*(Y)^{(H,\chi)}$. Then
$$WF(\xi) \subset \left\{ (x,v) \in T^*Y | v(T_x(Hx)) = 0 \right\}.$$
\end{thm}

\begin{thm}[{ \cite[Theorem 4.1.2]{A}}]\label{thm: wfs invariance to shifts}
Let $Y \subset X$ be $F$-analytic manifolds and let $y \in Y$. Let $\xi \in \cS^*(X)$ and suppose
that $\Supp(\xi) \subset Y$. Then $WF_{y}(\xi) $ is invariant with respect to shifts by the conormal space
$CN_{Y,y}^X$.
\end{thm}


\begin{cor}\label{cor:WCI}
Let $M$  be an $F$-analytic  manifold and $N\subset M$ be a closed algebraic submanifold. Let $\xi$ be a distribution on $M$ supported in $N$. Suppose that for any $x \in N$, we have $CN_{N,x}^M \nsubseteq WF_x(\xi)$. Then $\xi=0$.
\end{cor}
\begin{proof}
Suppose $\xi \ne 0$ and let $x \in \Supp(\xi)$. Then $(x,0) \in WF_x(\xi)$. But then from
Theorem \ref{thm: wfs invariance to shifts} we have $CN_{N,x}^M \subseteq WF_x(\xi)$ which contradicts our assumption on $\xi$.
\end{proof}

\subsection{Vanishing of equivariant distributions}
The following criterion for vanishing of equivariant distributions follows from \cite[\S 6]{BZ} and \cite[\S 1.5]{Ber}.

\begin{thm}[Bernstein-Gelfand-Kazhdan-Zelevinsky] \label{thm:BGKZ}
Let an algebraic group $\mathbf{H}$ act on an algebraic variety $\mathbf{X}$, both defined over $F$. Let $\chi$ be a character of $H$. 
Let $Z\subset X$ be a closed \Rami{$H$-invariant} subset.
Suppose that for any $x\in Z$
we have $$\chi|_{H_x}\neq\Delta_{H}|_{H_x}  \Delta_{H_x} ^{-1},$$
where $\Delta_{H}$ and $\Delta_{H_x}$ denote the modular functions of the groups $H$ and $H_x$.
Then there are no non-zero $(H,\chi)$-equivariant distributions on $X$ supported in $Z$.
\end{thm}

\DimaC{
\subsection{Characters of unipotent groups}
The following lemma is standard.
\begin{lem}\label{lem:char}
Let $\mathbf{V}$ be a unipotent algebraic group defined over $F$, let $\alpha$ be a (locally constant, complex) character of $V$ and $\beta$ be a non-trivial character of $F$. Then there exists an algebraic  group morphism $\phi:\mathbf{V} \to \mathbb{G}_a$ such that $\alpha=\beta \circ \phi$.
\end{lem}
For completeness we include a proof in Appendix \ref{App:char}. In the case when $\mathbf{V}$ is a maximal unipotent subgroup of a reductive group and $F$ is an arbitrary field (of an arbitrary characteristic) this lemma is  {\cite[Theorem 4.1]{BH}}. }

\section{Proof of Theorem \ref{thm: main theorem}}\label{sec:PfMain}
\begin{lem}\label{lem:WF}
Let $x \in G$. Let $\xi$ be a $(U,\psi)$-left equivariant and
$(H,\chi)$-right equivariant distribution on $G$ such that $WF(\xi) \subset G \times \NN_{\fg^*}$.
Then $WF_x(\xi) \subset CN^{G}_{BxH,x}$.
\end{lem}

\begin{proof}
Let $\ft$ be the Lie algebra of a maximal torus contained in $B$, and
let $\fh , \mathfrak{u}$ be the Lie algebras of $H,U$ respectively.
Identify $T^*_xG$ with $\fg^*$ using the right multiplication by $x^{-1}$. We have $CN_{BxH,x}^G= (\ft+\mathfrak{u}+ad(x)\fh)^\bot$.
Since $\xi$ is $\mathfrak{u}$-equivariant, by Theorem \ref{thm: equivar_distr} 
we have $WF_x(\xi)\subset \mathfrak{u}^\bot$. Similarly, since $\xi$ is $\fh$-equivariant on the right, we have
$WF_x(\xi) \subset (ad(x)\fh)^\bot$.
By our assumption $WF_x(\xi)\subset \cN_{\g^*} $. Now,
$\mathfrak{u}^\bot \cap \cN_{\g^*} = (\ft + \mathfrak{u})^\bot$ and thus
$$WF_x(\xi)\subset (ad(x)\fh)^\bot \cap \mathfrak{u}^\bot \cap \cN_{\g^*} =
(ad(x)\fh)^\bot \cap (\mathfrak{u}+\ft)^\bot =
(\ft+\mathfrak{u}+ad(x)\fh)^\bot = CN_{BxH,x}^G.$$
\end{proof}

%

\RamiA{
Now we would like to describe the structure of  the varieties $\cO_c$. For this we will use the following notation.
\begin{notn}
For a $B \times H$ double coset $\cO=BxH \subset G$ define $$\tilde \cO_c=\bigcup_{\cO'=ByH \subset (\mathbf{B}x  \mathbf{H})(F) } {\cO'_c}$$
\end{notn}
\begin{lem}\label{lem:AlgVar} For any  double coset $\cO=B x H\subset G$ there exists a closed algebraic subvariety  $\tilde{ \mathbf{O}}_c \subset \mathbf{B}x\mathbf{H}$ s.t.  $\tilde{\cO}_c=\tilde{\mathbf{O}}_c(F)$.
\end{lem}

\begin{proof}
Note that $$\cO_c=\left \{x \in \cO \, |\, \left. \psi^{x^{-1}} \right|_{H \cap U^{x^{-1}}} =  \left. \chi \right|_{H \cap U^{x^{-1}}} \right \}. $$

 Let $\mathbf H_x:= \mathbf H \cap \mathbf U^{x^{-1}}$. Since $\mathbf U$ is normal in $\mathbf B$, for any $y\in \mathbf  (\mathbf{B}x  \mathbf{H})(F)$ we have $\mathbf H_x= \mathbf H_y$. Thus we will denote $\mathbf H_{\cO}:=\mathbf H_x$.

By Lemma \ref{lem:char} there exist
an additive character $\beta$ of $F$ and algebraic group homomorphisms $\psi':\mathbf U \to \mathbb{G}_a$, $\chi':\mathbf H_{\cO}\to \mathbb{G}_a$ such that   $\psi=\beta \circ \psi'$ and $\chi|_{H_{\cO}}=\beta \circ \chi'$.\\
Let us show that $$\tilde \cO_c=\left\{y\in (\mathbf{B}x  \mathbf{H})(F)| \, \left. (\psi')^{y^{-1}}\right|_{\mathbf H_{\cO}} =   \chi'\right \}$$%
%
Indeed,  if $y \in \tilde \cO_c$  then $\beta\circ \left. (\psi')^{y^{-1}} \right|_{ H_{\cO}} =   \beta\circ \chi'$,  hence  $ \beta \circ (\chi'-(\psi')^{y^{-1}}|_{\mathbf H_{\cO}})=1$, thus $\chi'-(\psi')^{y^{-1}}|_{\mathbf H_{\cO}} $ is bounded on $H_{\cO}$, and thus $\chi'-(\psi')^{y^{-1}}|_{\mathbf H_{\cO}}$ is trivial. We obtain $$\tilde{ O}_c=\{y\in (\mathbf{B}x  \mathbf{H})(F) \, \vert \, \forall u\in  H_{\cO} \text{ we have} \, \psi'(yuy^{-1})=\chi'(u)\},$$
which  is clearly the set of $F$-points of a closed algebraic subvariety  of $\mathbf{B}x\mathbf{H}$.
\end{proof}}

\begin{cor}$\,$\label{cor:strat}
\begin{enumerate}
\item There exists a stratification of $\cO_c$ into a union of smooth $F$-analytic locally closed submanifolds $\cO_c^i$ s.t. $\bigcup_{i\leq i_0} \cO_c^i$ is open in $\cO_c$.
\item Moreover, if $\cO_c \neq \cO$ then the dimensions of $\cO_c^i$ are strictly smaller than the dimension of $\cO_c$. 

\end{enumerate}
\end{cor}

\begin{proof}[Proof of Theorem \ref{thm: main theorem}]
 Suppose that there exists a non-zero right $(U,\psi)$-equivariant and
 left $(H,\chi)$-equivariant distribution $\xi$  supported on $Z$
 such that $WF(\xi) \subset G \times \NN_{\fg^*}$.
 We decompose $G$ into $B\times H$-double cosets and prove the required vanishing coset by coset. For a $B\times H$-double coset $\cO\subset G$
 define $\cO_s:=\cO \setminus \cO_c$
 and stratify $\cO_{c}$\RamiA{, using Corollary \ref{cor:strat},}  to a union of smooth locally closed $F$-analytic subvarieties $\cO_c^i$.
The collection $$\{\cO_c^i \, | \, \cO\text{ is a }B\times H\text{-double coset}\}\cup \{\cO_s \, | \,\cO\text{ is a }B\times H\text{-double coset}\}$$ is a stratification of $G$. Order this collection to a sequence $\{S_i\}_{i=1}^N$ of smooth locally closed $F$-analytic submanifolds of $G$
such that $U_k:=\bigcup_{i=1}^k S_i$ is open in $G$ for any $1\leq k \leq N$.
Let $k$ be the maximal integer such that $\xi|_{U_{k-1}}=0$. Suppose $k\leq N$ and let $\eta:=\xi|_{U_{k}}$. Note that $\mathrm{Supp} (\eta) \subset S_k$.
We will now show that $\eta=0$, which leads to a contradiction.
\begin{enumerate}[{Case} 1.]
\item $S_k = \cO_s$ for some orbit $\cO$. Then  $\eta = 0$ by Theorem \ref{thm:BGKZ}
since $\eta$ is \\$(U \times H, \psi \times \chi)$-equivariant.

\item $S_k \subset \cO=\cO_c$ for some orbit $\cO$. Then $S_k\subset G \setminus Z$ and $\eta = 0$ by the conditions.
\item $S_{k}\subset \cO_c\subsetneq\cO$ for some orbit $\cO$.
In this case\RamiA{, by Corollary \ref{cor:strat},}
 $\dim \RamiA{S_k}< \dim \cO$ and thus $$CN_{S_{k},x}^G \supsetneq CN_{\cO,x}^G.$$
By Lemma \ref{lem:WF} we have, for any  $x \in S_k$, $$ WF_x(\eta)\subset CN_{\cO,x}^G \text{ and thus }CN_{S_k,x}^G \nsubseteq WF_x(\eta).$$
 By Corollary \ref{cor:WCI} this implies $\eta=0$.
 \end{enumerate}
\end{proof}


\section{Proof of \DimaA{Corollaries \ref{cor: galois case} and \ref{cor: group case}}} \label{section: proof of cors}


Let $\mathbf{U}'$
denote the derived group of $\mathbf{U}$.

\begin{lem}

\label{lem: stabilizer-group case}
Let
$\overline{W}$ be the Weyl group of $\mathbf{G}$. Let $\overline{w} \in \overline{W}$ and let $w\in G$ be \DimaB{ a representative of $\overline{w}$}. Suppose that
$w\mathbf{U}w^{-1} \cap \mathbf{U} \subset \mathbf{U}'$. Then $\overline{w}$ is  the longest  element in $\overline{W}$.
\end{lem}
\begin{proof}
\Rami{
Let $\gotU$ be the Lie algebra of $\mathbf{U}$. On the level of Lie algebras the condition
$wUw^{-1} \cap U \subset U'$ means that \Dima{$(Ad(w)\gotU) \cap \gotU \subset \gotU '$. The algebra $\gotU$ can
be decomposed as $$\gotU = \bigoplus_{\alpha \in \Phi^+} \gotG_{\alpha}.$$
It is easy to see that
$$(Ad(w)\gotU) \cap \gotU = \sum_{\alpha \in \Phi^+, \overline{w}^{-1}(\al) \in \Phi^+ }
\gotG_{\alpha}. $$
Let $\Delta \subset \Phi^+$ be the
set of simple roots in $\Phi^+$.
Then from the condition of the lemma
we obtain that 
$\overline{w}^{-1}(\Delta) \subset \Phi^-$, and as a consequence
$\overline{w}^{-1}(\Phi^+) \subset \Phi^-$.
Let $\overline{w}_0$ be the longest  element in $\overline{W}$. Then
$\overline{w}_0\overline{w}^{-1}(\Phi^+) \subset \Phi^+$.
Since $\Phi^+$ is a finite set and $\overline{w}_0\overline{w}^{-1}$ acts by an invertible
linear transformation, we get $\overline{w}_0\overline{w}^{-1}(\Phi^+) = \Phi^+$.  
Since  simple roots are the indecomposable ones, it follows that
$\overline{w}_0\overline{w}^{-1}(\Delta) = \Delta$. \Rami{This implies that $\overline{w}_0\overline{w}^{-1} = 1$
(see e.g. \cite[\S 10.3]{Hum}),}} and thus
$\overline{w}_0=\overline{w}$.}
\end{proof}

\begin{cor}\label{cor: orbit cond}
Let $\mathbf{H}$ be a reductive group. Assume $\mathbf{G} = \mathbf{H} \times \mathbf{H}$ and let $\Delta\mathbf{H}\subset \mathbf{G}$ be the diagonal copy of $\mathbf{H}$.
\DimaA{Denote} $\mathbf{X}=\mathbf{G}/\Delta\mathbf{H}$ and
let  $x\in X$ be such that $\mathbf{U}_x \subset \mathbf{U}'$. Then the orbit $\mathbf{B}x$ is open in $\mathbf{X}$.
\end{cor}

\begin{proof}
We can identify $\mathbf{X}$ with $\mathbf{H}$ using the projection on the first coordinate.
We can assume that $\mathbf{B}=\mathbf{B}_{\mathbf{H}}\times \mathbf{B}_{\mathbf{H}}$ where $\mathbf{B}_{\mathbf{H}}$ is a Borel subgroup of $\mathbf{H}$.  Let $\overline{W}$ be the Weyl group of $\mathbf{H}$ and $W$ be a set of its representatives. By the Bruhat decomposition,
$$ \mathbf{H}= \bigsqcup_{w \in W} \mathbf{B}_{\mathbf{H}}w \mathbf{B}_{\mathbf{H}}$$
It is well-known that the only open $\mathbf{B}_{\mathbf{H}}\times \mathbf{B}_{\mathbf{H}}$ orbit in $\mathbf{H}$ is $\mathbf{B}_{\mathbf{H}}w_0 \mathbf{B}_{\mathbf{H}}$, where $w_0 \in W$
is the representative \Dima{of} the longest Weyl  element. Let $ w \in W$. Let $\mathbf{U_H}$ be the nilradical of $\mathbf{B_H}$. Then
$$\mathbf{U}  _w =  \{ (u_1,u_2)| u_1wu_2 = w , \; \; u_1,u_2 \in \mathbf{U_{H}}  \}, $$
and we see that for a pair $(u_1,u_2) \in \mathbf{U} _w$ we have
$u_1 = w u_2 w^{-1} \in w\mathbf{U_{H}} w^{-1}$. Therefore,
$$\mathbf{U} _w \cong \mathbf{U_{H}}  \cap w\mathbf{U_{H}} w^{-1}.$$
Let $$R = \{ x \in \mathbf{X}\ | \mathbf{U} _x \subset \mathbf{U} '  \}  =
\{ x \in\mathbf{H} | \mathbf{U_{H}}   \cap x \mathbf{U_{H}}  x^{-1} \subset  \mathbf{U_{H}}'=[\mathbf{U_{H}},\mathbf{U_{H}}] \},$$
 and let $\mathbf{R}$ be the corresponding algebraic variety. Since $ \mathbf{U}  $ and $ \mathbf{U}  '$ are normal in $ \mathbf{B}  $, we obtain that
 $ \mathbf{R}  $ is $ \mathbf{B}   $-invariant.
The corollary 
follows now from
 Lemma \ref{lem: stabilizer-group case}.
\end{proof}

\begin{cor}\label{cor: orbit cond Galois case}
\DimaA{Let $\mathbf{H}\subset \mathbf{G}$ be a subgroup of Galois type. Then
for every non-open $B$-orbit $\cO \subset G/H$ there exists $y \in \cO$ such that $\psi(U_y) \ne 1$.}
\end{cor}
\begin{proof}
Let $\cO \subset \DimaA{G/H}$ be a non-open $B$-orbit and $x \in \cO$.
\DimaB{Then the map $\mathbf{B} \to \mathbf{G}$ given by the action on $x$ is not submersive and thus
 $\mathbf{B}x$ is not Zariski open in $\DimaA{\mathbf{\mathbf{G}/\mathbf{H}}}$. By Corollary \ref{cor: orbit cond} this implies
$\mathbf{U}_x \not \subset \mathbf{U}'$.} Thus, there exists a non-degenerate character $\varphi$ of $U$ such that
$\varphi(U_x) \ne 1$. For a fixed $x \in \cO$, the set of characters $\varphi'$ of $U$ such that
  $\varphi'(U_x) \ne 1$ is Zariski-open, thus dense in the $l$-space topology
and thus intersects the $B$-orbit of $\psi$. Thus there exists $y\in Bx=\cO$ such that
$\psi(U_y) \ne 1$.
\end{proof}
\DimaA{
\begin{proof}[Proof of Corollary \ref{cor: galois case}]
By Theorem \ref{thm: main theorem} it is enough to show that $S\subset Z$. Let $\cO \subset S$ be a $B\times H$ double coset. Corollary \ref{cor: orbit cond Galois case} implies that there exists $x \in \cO$ such that $\psi|_{U \cap H^x}\neq 1$. Since $H^x$ is reductive and $U$ is unipotent, we have $\chi^x|_{U \cap H^x}= 1$, and thus $\cO \subset Z$.
\end{proof}
 \begin{proof}[Proof of Corollary \ref{cor: group case}]
 Define $\mathbf{\tilde G} = \mathbf{G} \times \mathbf{G}$,
 $\mathbf{\tilde H} = \Delta(\mathbf{G}) \subset \mathbf{\tilde G}$ and $\mathbf{\tilde B}=\mathbf{B}\times \mathbf{B}$.  The non-degenerate characters $\psi_1,\psi_2$ define a non-degenerate character of the nilradical $\tilde U$ of $\tilde B$.
Note that
 $\mathbf{\tilde H} \subset \mathbf{\tilde G}$ is a subgroup of Galois type  and that $\tilde G/\tilde H $ is naturally isomorphic to $G$. Let $\eta$ be the pull-back of $\xi$ to $\tilde G$ under the projection $\tilde G \to \tilde G/\tilde H \cong G$. Then we have $\Supp \eta \subset \tilde S$, where $\tilde S$ is the union of all non-open $\tilde B\times \tilde H$-double cosets in $\tilde G$.
Also, by Corollary \ref{cor:SubPull} we have $WF(\eta) \subset \tilde G \times \NN_{\tilde \fg^*}$.
By Corollary \ref{cor: galois case} we obtain $\eta=0$ and thus $\xi=0$.
 \end{proof}
}
\begin{rem}
\DimaA{Corollary \ref{cor: galois case} can not be generalized literally to arbitrary symmetric pairs. The reason is that neither can Corollary \ref{cor: orbit cond Galois case}.} For example consider  the pair $\mathbf{G}=\mathbf{GL}_{2n},\mathbf{H}=\mathbf{GL_{n}}\times \mathbf{GL_{n}}$, where the involution is conjugation by the diagonal matrix with first $n$ entries equal to $1$ and others \Dima{equal to} $-1$. Let $x$  be the coset  of the permutation matrix given by the product of transpositions $$\prod_{i=0}^{\lfloor (n-1)/2 \rfloor}(2i+1,2n-2i),$$ and let  $\mathbf{B}$ consist of upper-triangular matrices. Then $\mathbf{U}_x\subset \mathbf{U}'$, while $\mathbf{B}x$ is of middle dimension in $\mathbf{G} /\mathbf{H}$. \DimaA{It can be shown that there exists a $(U,\psi)$-left equivariant, $H$-right invariant distribution $\xi$ on $G$ supported in $BxH$ and satisfying $WF(\xi)\subset G \times \NN_{\fg^*}$.}

However, Corollary \ref{cor:Bessel}\eqref{it:Gal} might hold for any spherical subgroup $\mathbf{H}$. In fact, this is the case over the archimedean fields, see \cite[Corollary B]{AG}.
\end{rem}

\appendix
\DimaC{
\section{Proof of Lemma \ref{lem:char}} \label{App:char}

\begin{lem}\label{lem:exp}
Let 
$\fv$ be the ($F$-points of) the Lie algebra of $\mathbf{V}$. Then the exponential map $\exp: \fv \to V$ maps the commutant $[\fv,\fv]$ of $\fv$ onto the subgroup $[V,V]$ of $V$ generated (set-theoretically) by all the commutators in $V$.
\end{lem}
\begin{proof}
Let $\fv_i$  be the sequence of subalgebras of $\fv$ defined by $\fv_0:=\fv, \, \fv_{i+1}:=[\fv,\fv_i]$.
The Baker-Campbell-Hausdorff formula implies that for any $X\in \fv$ and $Y\in \fv_i$ there exist $A,B\in \fv_{i+2}$ and $C \in \fv_{i+1}$ such that
\begin{align}
\log(e^Xe^Y) &= X+Y+\frac{1}{2}[X,Y]+A, \label{=BCH}\\
\log(e^Xe^Ye^{-X}e^{-Y}) &= [X,Y]+B \label{=cm}\\
e^{X+Y}&=e^{C}e^Xe^Y. \label{=ExpSum}
\end{align}
By (\ref{=BCH},\ref{=cm})   we have $[V,V]\subset\exp([\fv,\fv])$.
To prove the opposite inclusion we prove by descending induction on $i$ that $\exp(\fv_i)\subset [V,V]$ for any $i>0$. Since $\exp(\fv_i)$ is a group, it is enough to show that for any $X\in \fv$ and $Y\in \fv_{i-1}$ we have $\exp([X,Y])\in [V,V]$. Let $B$ be as in  \eqref{=cm}, and $C$ be  as in  \eqref{=ExpSum} applied to $\log(e^Xe^Ye^{-X}e^{-Y})$ and $-B$. Then $B,C\in \fv_{i+1}$, the induction hypothesis implies that $e^B,e^C\in [V,V]$ and thus
$$\exp([X,Y])=\exp(\log(e^Xe^Ye^{-X}e^{-Y})-B)= e^C(e^Xe^Ye^{-X}e^{-Y})e^{-B}\in [V,V].$$
\end{proof}

\begin{cor}
Let $\mathbf{V}/[\mathbf{V},\mathbf{V}]$ denote the abelization of $\mathbf V$. Then
the natural map $V/[V,V] \to (\mathbf{V}/[\mathbf{V},\mathbf{V}])(F)$ is an isomorphism.
\end{cor}
\begin{proof}
Let $\underline{\fv}$ be $\fv$ considered as an algebraic variety. By (\ref{=BCH},\ref{=cm}), the quotient
$\mathbf{V}/\exp([\underline{\fv},\underline{\fv}])$ is an abelian group. Hence  $[\mathbf{V},\mathbf{V}]\subset \exp([\underline{\fv},\underline{\fv}])$.  Thus, by Lemma \ref{lem:exp} we have
$[V,V]\subset  [\mathbf V, \mathbf V](F)\subset \exp([\fv,\fv])=[V,V].$
 Therefore $[V,V]=  [\mathbf V, \mathbf V](F)$. 
Since unipotent groups have trivial Galois cohomologies (see \cite[\S III.2.1, Proposition 6]{SerGal}), $\mathbf{V}(F)/[\mathbf{V},\mathbf{V}](F)=(\mathbf{V}/[\mathbf{V},\mathbf{V}])(F)$ and the statement follows.
\end{proof}
By this corollary Lemma \ref{lem:char} reduces to the case when $\mathbf{V}$ is commutative. Since any commutative unipotent group over $F$ is a power of $\mathbb G_a$, this case follows from the isomorphism of $F$ to its Pontryagin dual.
\proofend}

\end{document}